\documentclass[11pt]{amsart}

\pdfoutput=1

\usepackage[text={400pt,660pt},centering]{geometry}

\usepackage{esint,amssymb} 
\usepackage{graphicx}
\usepackage{MnSymbol}
\usepackage{mathtools} 
\usepackage[colorlinks=true, pdfstartview=FitV, linkcolor=blue, citecolor=blue, urlcolor=blue,pagebackref=false]{hyperref}
\usepackage{microtype}

\usepackage{bm}
\usepackage{dsfont}

\newtheorem{proposition}{Proposition}
\newtheorem{theorem}[proposition]{Theorem}
\newtheorem{lemma}[proposition]{Lemma}

\newtheorem{conjecture}[proposition]{Conjecture}
\theoremstyle{remark}

\theoremstyle{definition}

\numberwithin{equation}{section}
\numberwithin{proposition}{section}
\numberwithin{figure}{section}
\numberwithin{table}{section}
\newcommand{\N}{\mathbb{N}}

\newcommand{\R}{\mathbb{R}}

\newcommand{\E}{\mathbb{E}}
\renewcommand{\P}{\mathbb{P}}

\renewcommand{\le}{\leqslant}
\renewcommand{\ge}{\geqslant}

\renewcommand{\subset}{\subseteq}
\newcommand{\la}{\left\langle}
\newcommand{\ra}{\right\rangle}

\newcommand{\Ll}{\left}
\newcommand{\Rr}{\right}
\renewcommand{\d}{\mathrm{d}}

\DeclareMathOperator{\supp}{supp}

\renewcommand{\bar}{\overline}

\newcommand{\td}{\widetilde}

\newcommand{\1}{\mathds{1}}

\newcommand{\mcl}{\mathcal}

\newcommand{\al}{\alpha}
\newcommand{\be}{\beta}

\newcommand{\de}{\delta}
\newcommand{\si}{\sigma}

\newcommand{\dr}{\partial}

\newcommand{\vb}{\, \big \vert \, }

\renewcommand{\k}{^{(k)}}

\begin{document}

\author[J.-C. Mourrat]{J.-C. Mourrat}
\address[J.-C. Mourrat]{DMA, Ecole normale sup\'erieure,
CNRS, PSL University, Paris, France}
\email{mourrat@dma.ens.fr}

\date{\today}

\title[Parisi's formula is a Hamilton-Jacobi equation in Wasserstein space]{Parisi's formula is a Hamilton-Jacobi equation \\ in Wasserstein space}

\begin{abstract}
Parisi's formula is a self-contained description of the infinite-volume limit of the free energy of mean-field spin glass models. We show that this quantity can be recast as the solution of a Hamilton-Jacobi equation in the Wasserstein space of probability measures on the positive half-line.
\end{abstract}

\maketitle

%
%
%
%%%%%%%%%%%%%%%%%%%%%%%%%%%%
%%%%%%%%%%%%%%%%%%%%%%%%%%%%
%
%
%

\section{Introduction}

%\begin{quote}
%Many more facts which look like striking coincidences will occur. Of course, the author does not believe that they are mere coincidences, but rather that there is some underlying structure yet to be understood. 
%\end{quote}

%The goal of this article is to show that Parisi's formula for mean-field spin glasses can be recast as the solution of a Hamilton-Jacobi equation in a Wasserstein space of probability measures on $\R_+$. 

Let $(\beta_p)_{p \ge 2}$ be a sequence of non-negative numbers, which for simplicity we assume to contain only a finite number of non-zero elements, and, for every $r \in \R$, let $\xi(r) := \sum_{p \ge 2} \beta_p r^p$. For every integer $N \ge 1$, let~$P_N$ denote a probability measure on~$\R^N$, which we often (but not always) assume to be such that
\begin{equation}
\label{e.ass.sphere}
\Ll\{\begin{aligned}
& \mbox{either $P_N$ is the uniform measure on $\{-1,1\}^N$ for every $N \ge 1$,} 
\\
& \mbox{or $P_N$ is the uniform measure on $\{\si \in \R^N \ : \ |\si|^2 = N\}$ for every $N \ge 1$.}
\end{aligned}
\Rr.
\end{equation}
We aim to study Gibbs measures built from the probability measure $P_N$ using as energy function the centered Gaussian vector
$(H_N(\sigma))_{\sigma \in \R^N}$ with covariance
\begin{equation*}  %\label{e.}
\E \Ll[ H_N(\sigma) \, H_N(\tau) \Rr] = N \xi \Ll( \frac{\sigma \cdot \tau}{N} \Rr) \qquad (\sigma,\tau \in \R^N).
\end{equation*}
This Gaussian vector can be built explicitly using independent linear combinations of quantities of the form $\sum_{1 \le i_1,\ldots,i_p \le N} J_{i_1,\ldots,i_p} \sigma_{i_1} \cdots \, \sigma_{i_p}$, where $(J_{i_1,\ldots,i_p})$ are independent standard Gaussian random variables.
The Gibbs measures thus obtained are often called mixed $p$-spin models, possibly with the qualifiers ``spherical'' or ``with Ising spins'' when $P_N$ is the uniform measure on the sphere in $\R^N$ or on $\{-1,1\}^N$ respectively. The Sherrington-Kirkpatrick model corresponds to the case of Ising spins and $\xi(r) = \beta r^2$.
The \emph{Parisi formula} is a self-contained description of the limit free energy
\begin{equation*}  %\label{e.}
\lim_{N \to \infty} \frac 1 N \E \log  \int \exp \Ll( H_N(\sigma) \Rr) \, \d P_N(\sigma).
\end{equation*}
The identification of this limit was put on a rigorous mathematical footing in \cite{gue03, Tpaper,Tbook1, Tbook2, pan}, after the fundamental insights reviewed in \cite{MPV}. 

\smallskip

The main goal of the present paper is to propose a new way to think about this result. This new point of view reveals a natural connection with the solution of a Hamilton-Jacobi equation posed in the space of probability measures on the positive half-line.
 For every metric space $E$, we denote by $\mcl P(E)$ the set of Borel probability measures on $E$, and by $\de_x$ the Dirac measure at $x \in E$. We also define, with $\xi^*$ defined below in \eqref{e.def.xi*},
\begin{equation*}  %\label{e.}
\mcl P_*(\R_+) := \Ll\{ \mu \in \mcl P(\R_+) \ : \ \int_{\R_+} \xi^*(s) \, \d \mu(s) < \infty \Rr\} .
\end{equation*}
\begin{theorem}[Hamilton-Jacobi representation of Parisi formula]  
Assume \eqref{e.ass.sphere}, fix the normalization $\xi(1) = 1$, and let $(t,\mu) \mapsto f(t,\mu) : \R_+ \times \mcl P_*(\R_+) \to \R$ be the solution of the Hamilton-Jacobi equation
\label{t.main}
\begin{equation}  
\label{e.hj}
\Ll\{
\begin{aligned}
& \dr_t f - \int \xi(\dr_\mu f) \, \d \mu = 0 & \quad \text{on } \R_+ \times \mcl P_*(\R_+), 
\\
& f(0,\cdot) = \psi  &\quad \text{on } \mcl P_*(\R_+),
\end{aligned}
\Rr.
\end{equation}
where the function $\psi$ is described below in \eqref{e.def.psi} and Proposition~\ref{p.converge.psi}.
For every $t \ge 0$, 
\begin{equation}  
\label{e.main}
\lim_{N \to \infty} -\frac 1 N\E \log \int \exp \Ll( \sqrt{2t} H_N(\sigma) - N t\Rr)   \, \d P_N(\sigma) = f(t,\de_0). 
\end{equation}
\end{theorem}
Interestingly, the evolution equation in \eqref{e.hj} depends on the correlation function~$\xi$ but not on the measures $P_N$, while, as will be seen below, the initial condition~$\psi$ depends on the measures $P_N$ but not on $\xi$. We postpone a precise discussion of the meaning of the equation~\eqref{e.hj}, and start by explaining the background and motivations for looking for such a representation.

\smallskip

Recently, a new rigorous approach to the identification of limit free energies of mean-field disordered systems was proposed in \cite{HJinfer,HJrank}, inspired by \cite{gue01, barra1,barra2}.
The idea proposed there is to place the main emphasis on the fact that after ``enriching'' the problem, we can identify the limit free energy as the solution of a Hamilton-Jacobi equation. At least for the problems considered there, one can show that finite-volume free energies already satisfy the same Hamilton-Jacobi equation, except approximately. In particular, the approach allows for a convenient breakdown of the proof into two main steps: a first, more ``probabilistic'' part, which aims at showing that finite-volume free energies indeed satisfy an approximate Hamilton-Jacobi equation; and a second, more ``analytic'' part, which takes this information as input and concludes that the limit must solve the equation exactly. 

\smallskip

The problems studied in \cite{HJinfer,HJrank} relate to statistical inference. They possess a particular feature that enforces ``replica symmetry'', and this allows for a complete resolution of the problem by adding only a finite number of extra variables to the problem. As is well-known, this is not the case for mean-field spin glasses as those considered here. The relevant Hamilton-Jacobi equation, if any, must therefore be set in an infinite-dimensional space. 

\smallskip

The identity of this Hamilton-Jacobi equation is revealed by Theorem~\ref{t.main}. The aim of the present paper is to demonstrate the presence of this structure, and we will therefore simply borrow formulas from the literature for the limit on the left side of~\eqref{e.main}, and check that the expressions found there agree with the right side of~\eqref{e.main}. Hence, I want to stress that Theorem~\ref{t.main} is \emph{a rephrasing of known results}. 
%In particular, the proofs given in the present paper contain no serious technical innovation of mathematical analysis. 

\smallskip

However, I believe that Theorem~\ref{t.main} can be useful in furthering our understanding by providing a new way for us to think about these results---see also \cite{thurston} for general considerations on the relevance of such endeavors. In the long run, I hope indeed that this new interpretation of the Parisi formula will suggest a new and possibly more robust and transparent approach to the identification of the limit free energy of disordered mean-field systems. For this purpose, it will be important to rely on stability estimates for the Hamilton-Jacobi equation~\eqref{e.hj} (that is, estimates asserting that a function satisfying the equation approximately must be close to the true solution). This should leverage on powerful approaches to the well-posedness of Hamilton-Jacobi equations such as the notions of viscosity or weak solutions, as exemplified in the finite-dimensional setting in \cite{HJinfer} and \cite{HJrank} respectively. Since the purpose of the present paper is only to demonstrate the presence of the Hamilton-Jacobi structure, I will refrain from exploring this direction here and postpone it to future work. Indeed, since only exact identifications will be called for here, we can contend ourselves here with the more immediate definition of the solution of~\eqref{e.hj} based on the Hopf-Lax formulation.

\smallskip

This Hopf-Lax formulation features an optimal transport problem involving the cost function $(x,y) \mapsto \xi^*(x-y)$, where $\xi^*$ is the convex dual of $\xi$ defined by
\begin{equation}  
\label{e.def.xi*}
\xi^*(s) := \sup_{r \ge 0} (rs - \xi(r)).
\end{equation}
Notice that the function $\xi$ is convex on $\R_+$, and the precise way to interpret $\xi^*$ is as the dual of the convex and lower semicontinuous function on $\R$ which coincides with $\xi$ on $\R_+$ and is $+\infty$ otherwise.
(The functions $\bar F_N$ and the solution $f$ to \eqref{e.hj} share a monotonicity property which can be interpreted as $\dr_{\mu} f \ge 0$ in a weak sense, and thus modifying $\xi$ on $\R \setminus \R_+$ is irrelevant to the interpretation of \eqref{e.hj}---see also~\cite{HJrank} for a more precise discussion of this point in finite dimension, as well as Lemma~\ref{l.pos.dr_q} below).

\smallskip

Optimal transport problems for measures on the real line are in some sense trivial, in that the couplings between pairs of measures can be realized jointly over all measures, and do not depend on the convex function $\xi^*$ entering the definition of the cost function. Denoting, for every $\mu \in \mcl P(\R_+)$ and $r \in [0,1]$,
\begin{equation}
\label{e.def.F-1}
F^{-1}_\mu(r) := \inf \Ll\{ s \ge 0 \ : \ \mu \Ll( [0,s] \Rr)  > r \Rr\} ,
\end{equation}
and letting $U$ be a uniform random variable under $\P$, we set
\begin{equation}
\label{e.def.Xmu}
X_\mu := F^{-1}_\mu(U).
\end{equation}
It is classical to verify that the law of $X_\mu$ under $\P$ is $\mu$, and that for any two measures $\mu,\nu \in \mcl P_*(\R_+)$, the law of the pair $(X_\mu, X_\nu)$ is an optimal transport plan for the cost function $(x,y) \mapsto \xi^*(x-y)$ (see e.g.\ \cite[Theorem~2.18 and Remark~2.19(ii)]{villani} or \cite[Theorem~6.0.2]{AGS}). As discussed above, for the purposes of this paper we define the solution of \eqref{e.hj} to be given by the Hopf-Lax formula
\begin{equation}  
\label{e.hopflax}
f(t,\mu) := \sup_{\nu \in \mcl P_*(\R_+)} \Ll( \psi(\nu) - t\, \E \Ll[ \xi^* \Ll( \frac{ X_{\nu} - X_\mu }{t} \Rr)  \Rr] \Rr). 
\end{equation}
Although this will not be used here, one can give a brief and nonrigorous idea of the definition of the derivative $\dr_\mu$ formally appearing in \eqref{e.hj} in the case when it applies to a sufficiently ``smooth'' function $g : \mcl P_*(\R_+) \to \R$: for each $\mu \in \mcl P_*(\R_+)$, we want $\dr_\mu g(\mu,\cdot)$ to satisfy $\int \xi(\dr_\mu g(\mu,\cdot)) \, \d \mu < \infty$, and be such that, as $\nu \to \mu$ in $\mcl P_*(\R_+)$,
\begin{equation}  
\label{e.taylor.measure}
g(\nu) = g(\mu) + \E \Ll[ \dr_\mu g(\mu,X_\mu)(X_\nu - X_\mu) \Rr] + o \Ll( \Ll\| X_\nu - X_\mu \Rr\|_{L^*}  \Rr) ,
\end{equation}
where $\|Y\|_{L^*}$ denotes the $\xi^*$-Orlicz norm of a random variable $Y$, see \cite{RR}, 
\begin{equation*}  %\label{e.}
\|Y\|_{L^*} := \inf \Ll\{ t > 0 \ : \ \E \Ll[ \xi^*(t^{-1} Y) \Rr]  \le \xi^*(1) \Rr\} .
\end{equation*}
From this informal definition, one can work out finite-dimensional approximations of the equation \eqref{e.hj} by imposing, for instance, that only measures of the form $k^{-1} \sum_{\ell = 1}^k \de_{x_\ell}$ are ``permitted''. This brings us within the realm of finite-dimensional Hamilton-Jacobi equations and allows for instance to verify the correspondence between the equation \eqref{e.hj} and the Hopf-Lax formula \eqref{e.hopflax} at the level of these approximations.

\smallskip

We will in fact consider a richer family of finite-volume free energies than what appears on the left side of \eqref{e.main}, parametrized by $(t,\mu) \in \R_+ \times \mcl P_*(\R_+)$, and I expect that these free energies converge to $f(t,\mu)$ as $N$ tends to infinity, where $f$ is the solution of \eqref{e.hj}. In fact, I expect that a similar result holds for a much larger class of measures $P_N$ than those covered by the assumption of \eqref{e.ass.sphere}. A precise conjecture to this effect is presented in Section~\ref{s.conj}. The identification of the initial condition $\psi$ appearing in \eqref{e.hj} is then discussed in Section~\ref{s.conv.init}. The proof of Theorem~\ref{t.main} is given in Section~\ref{s.proof}. Finally, finite-dimensional approximations of \eqref{e.hj} are briefly explored in Section~\ref{s.finite.dim}.

%
%
%
%%%%%%%%%%%%%%%%%%%%%%%%%%%%
%%%%%%%%%%%%%%%%%%%%%%%%%%%%
%
%
%

\section{Conjecture for a general reference measure}
\label{s.conj}

The main goal of this section is to state a conjecture generalizing Theorem~\ref{t.main} to a wider class of measures $P_N$ than those appearing in \eqref{e.ass.sphere}. For simplicity, we retain the assumption that 
\begin{equation}
\label{e.ass.support}
\mbox{the measure $P_N$ is supported in the ball $\{\si \in \R^N \ : \ |\sigma|^2 \le N\}$.}
\end{equation}
If there exists some $R \in (0,\infty)$ such that for every $N$, the maesure $P_N$ is supported in the ball $\{\si \in \R^N \ : \ |\sigma|^2 \le R N\}$, then one can without loss of generality reduce to the case in \eqref{e.ass.support} by rescaling the function $\xi$.

\smallskip

In order to grow some familiarity with Theorem~\ref{t.main} and its conjectured generalization, we start by illustrating the driving idea in simpler settings. Possibly the simplest demonstration of the idea of identifying limit free energies of mean-field systems as solutions of Hamilton-Jacobi equations concerns the analysis of the Curie-Weiss model, see e.g.\ \cite[Section~1]{HJinfer} (earlier references include \cite{new86,bra83}). We give here another simple illustration for spin glasses in the high-temperature regime, which is similar to discussions in \cite{gue01}. For every $t, h \ge 0$, we consider the ``enriched'' free energy
\begin{multline}  
\label{e.def.FNcirc}
\bar F_N^\circ(t,h) 
\\
:= -\frac 1 N\E \log \int \exp \Ll( \sqrt{2t} H_N(\sigma) - Nt \xi \Ll( N^{-1} |\sigma|^2 \Rr) + \sqrt{2h} z \cdot \sigma - h|\sigma|^2\Rr) \, \d P_N(\sigma),
\end{multline}
where $z = (z_1,\ldots,z_N)$ is a vector of independent standard Gaussians, independent of~$H_N$, and where $|\sigma|^2 = \sum_{i = 1}^N \sigma_i^2$.
Notice that under the assumptions of Theorem~\ref{t.main}, this quantity equals $N$, and $\xi(N^{-1} |\sigma|^2) = 1$. The terms $-Nt \xi \Ll( N^{-1} |\sigma|^2 \Rr)$ and $-h|\sigma|^2$ inside the exponential in \eqref{e.def.FNcirc} are natural since they ensure that
\begin{equation*}  %\label{e.}
\E \Ll[ \exp \Ll( \sqrt{2t} H_N(\sigma) - Nt \xi \Ll( N^{-1} |\sigma|^2 \Rr)  + \sqrt{2h} z \cdot \sigma - h|\sigma|^2\Rr)  \Rr] = 1.
\end{equation*}
(Observing that $H_N(\sigma)$ and $z\cdot \sigma$ are independent centered Gaussians of variance $N \xi \Ll( N^{-1} |\sigma|^2 \Rr)$ and $|\sigma|^2$ respectively, this follows either by recognizing an exponential martingale, or by differentiating in~$t$ and $h$ and using Gaussian integration by parts.) In statistical physics' terminology, one may say that we have normalized the Hamiltonian so that the annealed free energy is always zero. The minus sign in front of the expression on the right side of \eqref{e.def.FNcirc} is also convenient since, by Jensen's inequality, we thus have $\bar F_N^\circ \ge 0$. One can check that
\begin{equation}  
\label{e.hj.RS}
\dr_t \bar F_N^\circ - \xi \Ll( \dr_h \bar F_N^\circ \Rr) = \E \la \xi \Ll( \frac{\sigma \cdot \sigma'}{N} \Rr) \ra - \xi \Ll( \E \la  \frac{\sigma \cdot \sigma'}{N} \ra \Rr) .
\end{equation}
In the case when $\xi$ is convex over $\R$, the right side of \eqref{e.hj.RS} is non-negative, and thus we already see that $\bar F_N^\circ$ is a supersolution of a simple Hamilton-Jacobi equation. Moreover, one can expect in many settings that the initial condition $\bar F_N^\circ(0,h)$ converges as $N$ tends to infinity; for instance, when $P_N$ is the $N$-fold product measure $P_N = P_1^{\otimes N}$, we have 
\begin{equation*}  %\label{e.}
\bar F_N^\circ(0,h) = \bar F_1^\circ(0,h) = -\E \log \int_\R \exp \Ll( \sqrt{2h} \, z_1  \sigma  - h \sigma^2 \Rr)  \, \d P_1(\sigma),
\end{equation*}
where in this expression, the variable $z_1$ is a scalar standard Gaussian.
Finally, if we expect the overlap $\sigma \cdot \sigma'$ to be concentrated around its expectation, which should be correct in a high-temperature region (that is, for $t$ sufficiently small), then it should be that $\bar F_N^\circ$ converges to the solution of the equation $\dr_t f - \xi(\dr_h f) = 0$. 

\smallskip

However, as is well-known, the overlap $\sigma \cdot \sigma'$ is in fact not always concentrated around its mean value, and a more refined approach is necessary. In order to proceed, as in \cite{gue03,Tpaper,Tbook1, Tbook2, pan}, we need to compare the system of interest with a much more refined ``linear'' system than $\sqrt{2h} z \cdot \sigma$. We parametrize the more refined systems by a measure $\mu \in \mcl P(\R_+)$ (and not $\mu \in \mcl P([0,1])$ as experts may expect).
It is much more convenient to describe this more refined system in the case when $\mu$ is a measure of finite support: we assume that for some integer $k \ge 0$, there exist 
\begin{equation}  
\label{e.def.zeta.q}
0 = \zeta_{0} < \zeta_1 < \cdots < \zeta_{k} < \zeta_{k+1} = 1, \qquad 0 = q_{-1} \le q_0 < q_1 < \cdots < q_k < q_{k+1} =  \infty
\end{equation}
such that
\begin{equation}  
\label{e.mu.diracs}
\mu = \sum_{\ell = 0}^{k} (\zeta_{\ell+1} - \zeta_{\ell}) \de_{q_\ell}.
\end{equation}
We represent the rooted tree with (countably) infinite degree and depth $k$ by
\begin{equation*}  %\label{e.}
\mcl A := \N^{0} \cup \N \cup \N^2 \cup \cdots \cup \N^k,
\end{equation*}
where $\N^{0} = \{\emptyset\}$, and $\emptyset$ represents the root of the tree. For every $\alpha \in \N^\ell$, we write $|\alpha| := \ell$ to denote the depth of the vertex $\alpha$ in the tree $\mcl A$. For every leaf $\alpha = (n_1,\ldots,n_k)\in \N^k$ and $\ell \in \{0,\ldots, k\}$, we write
\begin{equation*}  %\label{e.}
\alpha_{| \ell} := (n_1,\ldots, n_\ell),
\end{equation*}
with the understanding that $\alpha_{| 0} = \emptyset$.
%we denote the set of predecessors of $\alpha$, excluding the root, by
%\begin{equation*}  %\label{e.}
%p(\alpha) := (n_1, (n_1,n_2),\ldots, (n_1,\ldots,n_k)).
%\end{equation*}
We also give ourselves a family $(z_{\al,i})_{\alpha \in \mcl A, 1 \le i \le N}$ of independent standard Gaussians, independent of $H_N$, and we let $(v_\alpha)_{\al \in \N^k}$ denote a Poisson-Dirichlet cascade with weights given by the family $(\zeta_\ell)_{1 \le \ell \le k}$. We refer to \cite[(2.46)]{pan} for a precise definition, and briefly mention here the following three points. First, in the case $k = 0$, we simply set $v_{\emptyset} = 1$. Second, in the case $k = 1$, the weights $(v_\alpha)_{\al \in \N}$ are obtained by normalizing a Poisson point process on $(0,\infty)$ with intensity measure~$\zeta_1 x^{-1-\zeta_1} \, \d x$ so that $\sum_{\alpha} v_\alpha = 1$. Third, for general~$k \ge 1$, the progeny of each non-leaf vertex at level $\ell \in \{0,\ldots, k-1\}$ is decorated with the values of an independent Poisson point process of intensity measure $\zeta_{\ell+1} x^{-1-\zeta_{\ell+1}} \, \d x$, then the weight of a given leaf $\alpha \in \N^k$ is calculated by taking the product of the ``decorations'' attached to each parent vertex, including the leaf vertex itself (but excluding the root), and finally, these weights over leaves are normalized so that their total sum is~$1$. We take this Poisson-Dirichlet cascade $(v_\alpha)_{\alpha \in \N^k}$ to be independent of $H_N$ and of the random variables $(z_\alpha)_{\al \in \mcl A}$. For every $\sigma \in \R^N$ and $\al \in \N^k$, we set
\begin{equation}
\label{e.def.HNalpha}
H_N'(\sigma,\alpha) := \sum_{\ell = 0}^k \Ll(2q_{\ell} - 2q_{\ell-1}\Rr)^\frac 1 2 z_{\alpha_{|\ell}} \cdot \sigma,
\end{equation}
where we write $z_{\al_{|\ell}} \cdot \sigma = \sum_{i = 1}^N z_{\al_{|\ell},i} \, \sigma_i$. The random variables $(H_N'(\sigma,\alpha))_{\sigma \in {\R^N}, \alpha \in \N^k}$ form a Gaussian family which is independent of $(H_N(\sigma))_{\sigma \in \R^N}$ and has covariance
\begin{equation*}  %\label{e.}
\E \Ll[ H_N'(\sigma,\alpha) \, H_N'(\tau,\beta) \Rr] = 2q_{\alpha \wedge \beta} \  \sigma \cdot \tau \qquad (\sigma,\tau \in \R^N, \ \alpha, \beta \in \N^k),
\end{equation*}
where we write, for every $\alpha, \beta \in \N^k$,
\begin{equation*}  %\label{e.}
\alpha \wedge \beta := \sup \{\ell \le k \ : \ \al_{|\ell} = \be_{|\ell} \}.
\end{equation*}
We define the ``enriched'' free energy as
\begin{multline}
\label{e.def.FN}
F_N(t,\mu) := -\frac 1 N \log \int \sum_{\al \in \N^k}  \exp \Big( \sqrt{2t} H_N(\sigma) - N t\xi \Ll( N^{-1} |\si|^2 \Rr)  
\\
+ H_N'(\sigma,\alpha) - q_k|\sigma|^2  \Big) \,v_\al \,  \d P_N(\sigma),
\end{multline}
and $\bar F_N(t,\mu) := \E \Ll[ F_N(t,\mu) \Rr]$.
As for \eqref{e.def.FNcirc}, we have normalized this expression so that, by Jensen's inequality, we have $\bar F_N \ge 0$. We first notice that this quantity can be extended to all $\mu \in \mcl P_1(\R_+)$ by continuity.
\begin{proposition}[Continuity and extension of $\bar F_N(t,\mu)$]
\label{p.continuity}
Assume \eqref{e.ass.support}. 
For each $t \ge 0$ and $\mu, \mu' \in \mcl P(\R_+)$ with finite support, we have
\begin{equation}  
\label{e.continuity}
\Ll| \bar F_N(t,\mu') - \bar F_N(t,\mu) \Rr| \le \E \Ll[ |X_{\mu'} - X_{\mu}| \Rr] .
\end{equation}
In particular, the mapping $\mu \mapsto \bar F_N(t,\mu)$ can be extended by continuity to the set
\begin{equation*}  %\label{e.}
\mcl P_1(\R_+) := \Ll\{ \mu \in \mcl P(\R_+) \ : \ \int s \, \d \mu(s) < \infty \Rr\} .
\end{equation*}

\end{proposition}
%Notice that if there exists $R \in (0,\infty)$ such that for every $N$, the measure $P_N$ is supported in the ball $\{\sigma \in \R^N \ : \ |\sigma|^2 \le RN\}$, then we can replace $R$ by $1$ without loss of generality by rescaling the function $\xi$ and applying a dilation to the measure~$\mu$. Alternatively, one can prove Proposition~\ref{p.continuity} under this more general assumption, provided that the right side of \eqref{e.continuity} is multiplied by $R$. 
%
%\smallskip

The proof of proposition~\ref{p.continuity} makes use of the following two lemmas. The first one provides an explicit procedure for integrating the randomness coming from the Poisson-Dirichlet cascade. We refer to \cite[Theorem~2.9]{pan} for a proof. (Notice that the indexation of the family $\zeta$ differs by one unit between here and \cite{pan}.)
\begin{lemma}[Integration of Poisson-Dirichlet cascades]
Assume \eqref{e.ass.support}, and fix $t \ge 0$. For every $y_0, \ldots,$ $y_k \in \R^N$, define
\label{l.integr.PD}
\begin{multline}  
\label{e.def.Xk}
X_k(y_0,\ldots,y_k) := 
\\
\log \int \exp \Ll( \sqrt{2t} H_N(\sigma) - N t\xi \Ll( N^{-1} |\si|^2 \Rr)   + \sum_{\ell = 0}^k \Ll(2q_{\ell} - 2q_{\ell-1}\Rr)^\frac 1 2 y_\ell \cdot \sigma - q_k |\sigma|^2  \Rr) \, \d P_N(\sigma),
\end{multline}
and then recursively, for every $\ell \in \{1,\ldots, k\}$,
\begin{equation}  
\label{e.def.Xell}
X_{\ell-1}(y_0,\ldots,y_{\ell-1}) := \zeta_{\ell}^{-1} \log \E_{y_{\ell}} \exp \Ll( \zeta_{\ell} X_{\ell}(y_0,\ldots,y_{\ell}) \Rr) ,
\end{equation}
where, for every $\ell \in \{0,\ldots,k\}$, we write $\E_{y_{\ell}}$ to denote the integration of the variable $y_{\ell} \in \R^N$ along the standard Gaussian measure.
We have
\begin{equation*}  %\label{e.}
-N \, \E \Ll[ F_N(t,\mu)  \vb   (H_N(\sigma))_{\sigma \in \{\pm 1\}^N} \Rr] = \E_{y_0}\Ll[X_0(y_0)\Rr].
\end{equation*}
\end{lemma}

In order to state the second lemma, we introduce notation for the Gibbs measure associated with the free energy $F_N$. That is, for every bounded measurable function $f : \R^N \times \N^k \to \R$, we write
\begin{multline}  
\label{e.def.gibbs}
\la f(\sigma,\alpha) \ra_{t,\mu} := \exp \Ll( N F_N(t,\mu) \Rr) 
\\
\times \int \sum_{\al \in \N^k} f(\sigma,\alpha) \exp \Ll( \sqrt{2t} H_N(\sigma) - N t\xi \Ll( N^{-1} |\si|^2 \Rr)  
+ H_N'(\sigma,\alpha) - q_k|\sigma|^2  \Rr) \, v_\al \, \d P_N(\sigma).
\end{multline}
We usually simply write $\la \cdot \ra$ instead of $\la \cdot \ra_{t,\mu}$ unless there is a risk of confusion. Notice that the measure $\la \cdot \ra$ depends additionally on the realization of the Gaussian field $(H_N(\sigma))$ and of the variables $(z_{\al})$. By the definition of $F_N(t,\mu)$, we have $\la 1 \ra = 1$, and thus~$\la \cdot \ra$ can be interpreted as a probability distribution on $\R^N \times \N^k$. We also need to consider ``repliced'' pairs, denoted by $(\sigma,\alpha), (\sigma',\alpha'), (\sigma'', \al''), \ldots$, which are independent and are each distributed according to $\la \cdot \ra$ (conditionally on $(H_N(\si))$ and~$(z_\al)$). We keep writing $\la \cdot \ra$ to denote the tensorized measure, so that for instance, for every bounded measurable $f,g : \R^N \times \N^k \to \R$, we have
\begin{equation*}  %\label{e.}
\la f(\sigma,\alpha) \, g(\sigma',\alpha') \ra = \la f(\sigma,\alpha) \ra \, \la g(\sigma',\alpha') \ra .
\end{equation*}
The second lemma we need identifies the law of the overlap between $\al$ and $\al'$ under the Gibbs measure, after also averaging over the randomness coming from $(H_N(\sigma))$ and $(z_\al)$ (averaging over $(z_\al)$ only would be sufficient).
\begin{lemma}[overlaps for the Poisson-Dirichlet variables]
\label{l.overlap.pd}
Whenever the measure $\mu$ is of the form in \eqref{e.def.zeta.q}-\eqref{e.mu.diracs}, we have, for every $t \ge 0$ and $\ell \in \{0,\ldots,k\}$,
\begin{equation*}  %\label{e.}
\E \la \1_{\{\al \wedge \al'=\ell\}} \ra_{t,\mu} = (\zeta_{\ell+1} - \zeta_{\ell}).
\end{equation*}
\end{lemma}
\begin{proof}
The argument can be extracted from \cite{Tbook2}, or by observing that the derivation of \cite[(2.82)]{pan} applies as well to the measures considered here. A slightly adapted version of the latter argument is as follows. We fix $\ell \in \{0,\ldots, k\}$, and let $(g_\beta)_{\beta \in \N^\ell}$ be a family of independent standard Gaussians, independent of any other random variable considered so far. For every $\al, \be \in \N^k$, we have
\begin{equation}  
\label{e.correl}
\E \Ll[ g_{\al|\ell} \, g_{\be|\ell} \Rr] = \1_{\{ \al_{|\ell} = \be_{|\ell} \}}.
\end{equation}
Recall the construction of the Poisson-Dirichlet cascade outlined in the paragraph above \eqref{e.def.HNalpha}, see also \cite[(2.46)]{pan}, and denote by $w_\al$ the weights attributed to the leaves by taking the product of the ``decorations'' of the parent vertices, before normalization, as in \cite[(2.45)]{pan}, so that 
\begin{equation*}  %\label{e.}
v_\alpha = \frac{w_\al}{\sum_{\beta \in \N^k} w_{\be}}.
\end{equation*}
By \cite[(2.26)]{pan}, for every $s \in \R$, we have that
\begin{equation*}  %\label{e.}
(w_\alpha, g_{\al_{|\ell}})_{\al \in \N^k} \quad \text{ and } \quad \Ll(w_\al \exp \Ll( s \Ll(g_{\al_{|\ell}} - \tfrac {s\zeta_\ell}{2} \Rr) \Rr),  g_{\al_{|\ell}} - s\zeta_\ell\Rr)_{\al \in \N^k}
\end{equation*}
have the same law up to reorderings that preserve the tree structure: that is, we identify two families $(a_\al)_{\al \in \N^k}$ and $(b_\al)_{\al \in \N^k}$ whenever there exists a bijection $\pi : \N^k \to \N^k$ satisfying, for every $\al, \be \in \N^k$,
\begin{equation*}  %\label{e.}
a_{\al} = b_{\pi(\al)} \quad \text{and} \quad \pi(\al) \wedge \pi(\be) = \al \wedge \be.
\end{equation*}
We denote 
\begin{equation*}  %\label{e.}
v_{\alpha,\ell,s} := \frac{w_\al\exp \Ll( s g_{\al_{|\ell}} \Rr)}{\sum_{\beta \in \N^k} w_{\be}\exp \Ll( s g_{\be_{|\ell}} \Rr)},
\end{equation*}
and write $\la \cdot \ra_{t,\mu,\ell,s}$ to denote the measure defined as in \eqref{e.def.gibbs} but with $v_\al$ replaced by~$v_{\al,\ell,s}$. By the invariance described above, Gaussian integration by parts, and~\eqref{e.correl}, we have, for every $s \in \R$, 
\begin{equation*}  %\label{e.}
0 = \E \la g_{\al_{|\ell}} \ra_{t,\mu,\ell,0} = \E \la g_{\al_{|\ell}} - s\zeta_\ell \ra_{t,\mu,\ell,s}  = s \, \E \la 1-\1_{\{\al_{|\ell} = \al'_{|\ell}\}} - \zeta_\ell \ra_{t,\mu,\ell,s},
\end{equation*}
and, using the invariance once more, we can replace $\la \cdot \ra_{t,\mu,\ell,s}$ by $\la \cdot \ra_{t,\mu}$ in the last expression. We thus conclude that
\begin{equation*}  %\label{e.}
\E \la \1_{\{\al_{|\ell} = \al'_{|\ell}\}} \ra_{t,\mu} = 1-\zeta_\ell,
\end{equation*}
which yields the desired result.
\end{proof}
%The argument above can be adapted to yield a proof of some GG-type identities: for instance, we can get
%\begin{equation*}  %\label{e.}
%\E \la (\si \cdot \si') \1_{\{ \al = \al'' \}} \ra = \frac 1 2 \E \la \si \cdot \si' \ra (1-\zeta_\ell) + \frac 1 2 \E \la (\si \cdot \si') \1_{\{\al = \al'\}} \ra.
%\end{equation*}

\begin{proof}[Proof of Proposition~\ref{p.continuity}] 
We decompose the proof into two steps.

\smallskip

\emph{Step 1.} In this step, we give a consistent extension of the definition of $\bar F_N(t,\mu)$ to the case when the parameters in \eqref{e.def.zeta.q} may contain repetitions. More precisely, we give ourselves possibly repeating parameters
\begin{equation}  
\label{e.rep.zeta.q}
0 = \zeta_{0} \le \zeta_1 \le  \cdots \le \zeta_{k} \le \zeta_{k+1} = 1, \qquad 0 = q_{-1} \le q_0 \le \cdots \le q_k < q_{k+1} =  \infty,
\end{equation}
and let $\mu$ be the measure defined by \eqref{e.mu.diracs}. We show that the naive extension of the definition of $\bar F_N(t,\mu)$ obtained by simply ignoring the fact that there may be repetitions in the parameters in \eqref{e.rep.zeta.q} yields the same result as the actual definition that was given using non-repeating parameters. The first thing we need to do is extend the definition of the Poisson-Dirichlet cascade $(v_\alpha)_{\alpha \in \N^k}$ to the case when some values of $(\zeta_\ell)_{\ell \in \{1,\ldots,k\}}$ may be equal to $0$ or to $1$. Recall that for $\zeta_\ell \in (0,1)$, the definition briefly described in the paragraph above \eqref{e.def.HNalpha} involves a Poisson point process of intensity measure $\zeta_{\ell} x^{-1-\zeta_\ell} \, \d x$. In the case $\zeta_\ell = 0$, we interpret this Poisson point process as consisting of a single instance of the value $1$ and then a countably infinite repetition of the value $0$. In the case when $\zeta_\ell = 1$, we interpret this Poisson point process as consisting of a countably infinite repetition of the value $1$. This allows to define the quantity on the right side of \eqref{e.def.FN} for arbitrary values of the parameters in \eqref{e.rep.zeta.q}. The average of this quantity can be calculated using Lemma~\ref{l.integr.PD}: the only point that needs to be added is that in the case $\zeta_\ell = 0$, we interpret \eqref{e.def.Xell} as 
\begin{equation*}  %\label{e.}
X_{\ell-1}(y_0,\ldots,y_{\ell-1}) := \E_{y_{\ell}} \Ll[ X_{\ell}(y_0,\ldots,y_{\ell}) \Rr] .
\end{equation*}
From this algorithmic procedure, one can check that the result does not depend on whether or not there were repetitions in the parameters in \eqref{e.rep.zeta.q}. Indeed, on the one hand, when $\zeta_\ell = \zeta_{\ell+1}$, we have
\begin{equation*}  %\label{e.}
X_{\ell-1} (y_0,\ldots,y_{\ell-1}) = \zeta_{\ell}^{-1} \log \E_{y_\ell,y_{\ell+1}} \exp \Ll( \zeta_\ell X_{\ell+1}(y_0,\ldots,y_{\ell+1}) \Rr) ,
\end{equation*}
where $\E_{y_\ell,y_{\ell+1}}$ denotes the averaging of the variables $y_\ell, y_{\ell+1}$ when sampled independently according to the standard Gaussian measure on $\R^N$; and under this measure, the sum
\begin{equation*}  %\label{e.}
\Ll[(2q_\ell - 2q_{\ell-1})^\frac 1 2 y_\ell + (2q_{\ell+1} - 2q_{\ell})^\frac 1 2 y_{\ell+1}\Rr] \cdot \si
\end{equation*}
has the same law as
\begin{equation*}  %\label{e.}
(2q_{\ell+1} - 2q_{\ell-1})^\frac 1 2 y_\ell \cdot \si.
\end{equation*}
On the other hand, if $q_\ell = q_{\ell+1}$, then the term indexed by $\ell+1$ in the sum on the right side of \eqref{e.def.Xk} vanishes, and 
\begin{equation*}  %\label{e.}
X_{\ell} (y_0,\ldots,y_{\ell}) = X_{\ell+1}(y_0,\ldots,y_{\ell+1}) .
\end{equation*}
It is thus clear in both cases that removing repetitions does not change the value of the resulting quantity.

\smallskip

\emph{Step 2.} Consider now two measures $\mu,\mu' \in \mcl P(\R_+)$ of finite support. There exist $k \in \N$, $(\zeta_\ell)_{0 \le \ell \le k}$, $(q_\ell)_{0 \le \ell \le k}$ and $(q'_\ell)_{0 \le \ell \le k}$ satisfying \eqref{e.mu.diracs}, \eqref{e.rep.zeta.q}, 
\begin{equation*}  %\label{e.}
0 = q_{-1}' \le q_0' \le q_1' \le \cdots \le q_k' < q'_{k+1} =  \infty, \quad \text{and} \quad \mu' = \sum_{\ell = 0}^k (\zeta_{\ell+1} - \zeta_{\ell}) \de_{q'_\ell}.
\end{equation*}
Using this representation, we can rewrite the $L^1$-Wasserstein distance between the measures $\mu$ and $\mu'$ as
\begin{equation}  
\label{e.L1.wasserstein}
\E \Ll[ | X_{\mu'} - X_{\mu}|\Rr] = \sum_{\ell = 0}^k \Ll( \zeta_{\ell+1} - \zeta_{\ell}  \Rr) |q'_\ell - q_\ell|.
\end{equation}
Abusing notation, we denote
\begin{equation}  
\label{e.barFN.zq}
\bar F_N(t,\zeta,q) := \bar F_N \Ll( t,  \sum_{\ell = 0}^k (\zeta_{\ell+1} - \zeta_{\ell}) \de_{q_\ell}\Rr),
\end{equation}
and proceed to compute $\dr_{q_\ell} \bar F_N(t,\zeta,q)$, for each $\ell \in \{0,\ldots,k\}$.
For every $\sigma, \tau \in \R^N$, $\al \in \N^k$ and $\beta \in \mcl A$, we have
\begin{equation}  
\label{e.correl.H'z}
\E \Ll[ H_N'(\si,\al) z_{\beta} \cdot \tau\Rr] = 
\Ll|
\begin{array}{ll}
 (2q_{\ell} - 2 q_{\ell-1})^\frac 1 2  \, \si \cdot \tau \ & \text{if } \beta = \al_{| \ell} \text{ with } \ell \in \{0,1,\ldots,k\}, \\
 0 & \text{otherwise}.
\end{array}
\Rr.
\end{equation}
For every $\ell \in \{0,\ldots,k-1\}$, we have 
\begin{equation*}  %\label{e.}
\dr_{q_\ell} \bar F_N(t,\zeta,q) = -\frac 1 N \E \la (2q_\ell - 2q_{\ell - 1})^{-\frac 1 2} z_{\al_{|\ell}}\cdot \si - (2q_{\ell+1} - 2q_{\ell})^{-\frac 1 2} z_{\al{|(\ell+1)}}\cdot \si \ra.
\end{equation*}
By \eqref{e.correl.H'z} and Gaussian integration by parts, see e.g.\ \cite[Lemma~1.1]{pan}, we obtain
\begin{align}  
\label{e.expr.drq}
\dr_{q_\ell} \bar F_N(t,\zeta,q) 
& 
= \frac 1 N \E \la \Ll(\1_{\{\al_{|\ell} \, = \, \al'_{|\ell}\}} - \1_{\{\al_{|(\ell+1)}  =  \al'_{|(\ell+1)}\}} \Rr)\sigma \cdot \sigma' \ra
\\
\notag
& = 
\frac 1 N\E \la \1_{\{\al \wedge \al'=\ell\}} \, \sigma \cdot \sigma' \ra.
\end{align}
The same reasoning also shows that
\begin{equation*}  %\label{e.}
\dr_{q_k} \bar F_N(t,\zeta,q) = \frac 1 N \E \la \1_{\{\al = \al'\}} \si \cdot \si' \ra,
\end{equation*}
so that the last identity in \eqref{e.expr.drq} is also valid for $\ell = k$. 
In particular, for every $\ell \in \{0,\ldots,k\}$, we have by \eqref{e.ass.support} that
\begin{equation*}  %\label{e.}
\Ll|\dr_{q_\ell} \bar F_N(t,\zeta,q) \Rr|\le  \E \la \1_{\{\al \wedge \al'=\ell\}} \ra = (\zeta_{\ell+1} - \zeta_{\ell}),
\end{equation*}
and thus, by integration,
\begin{equation*}  %\label{e.}
\Ll|\bar F_N(t,\zeta, q') - \bar F_N(t,\zeta,q)\Rr| \le \sum_{\ell = 0}^k (\zeta_{\ell+1} - \zeta_\ell) \Ll| q'_{\ell} - q_\ell \Rr|.
\end{equation*}
A comparison with \eqref{e.L1.wasserstein} then yields the desired result.
\end{proof}
We can also use Lemma~\ref{l.integr.PD} to give a more precise meaning to the vaguely stated monotonicity claim of $\dr_{\mu} f \ge 0$ expressed in the paragraph below \eqref{e.def.xi*}, already at the level of the functions $\bar F_N$. 
\begin{lemma}
\label{l.pos.dr_q}
Let $\zeta,q$ be parameters as in \eqref{e.def.zeta.q}, and let $\bar F_N(t,\zeta,q)$ be as in \eqref{e.barFN.zq}. For every $\ell \in \{0,\ldots, k\}$, we have
\begin{equation}
\label{e.pos.dr_q}
\dr_{q_\ell} \bar F_N(t,\zeta,q) \ge 0.
\end{equation}
\end{lemma}
\begin{proof}
The proof is similar to that of \cite[Proposition~14.3.2]{Tbook2}; see also \cite{gue01,pantal07, barra1,barra2}. We will rewrite the left side of \eqref{e.pos.dr_q} as an averaged overlap, taking Lemma~\ref{l.integr.PD} as a starting point, the subtle point being in the identification of the correct measure with respect to which the average is taken. We start by introducing some notation. We let $X_k,X_{k-1},\ldots,X_0$ be as in Lemma~\ref{l.integr.PD}, and define $X_{-1} := \E_{y_0} \Ll[ X_0(y_0) \Rr]$. For every $\ell \le m \in \{0,\ldots,k\}$, we write
\begin{equation*}  %\label{e.}
D_{\ell,m} = D_{\ell m} := \frac{\exp \Ll( \zeta_{\ell} X_\ell + \cdots + \zeta_{k} X_m \Rr)}{\E_{y_\ell} \Ll[ \exp \Ll( \zeta_\ell X_\ell \Rr) \Rr] \, \cdots \ \E_{y_m} \Ll[ \exp \Ll( \zeta_\ell X_m \Rr) \Rr] }.
\end{equation*}
We also write $\E_{y_{\ge \ell}}$ to denote the integration of the variables $y_{\ell}, \ldots, y_k$ along the standard Gaussian measure, and we write $\E_y$ as shorthand for $\E_{y_{\ge 0}}$. Within the current proof (and only here), we abuse notation and use $\la \cdot \ra$ with a meaning slightly different from that in \eqref{e.def.gibbs}, namely,
\begin{multline*}  %\label{e.}
\la f(\si) \ra := \exp \Ll( -X_k \Rr) 
\\
\times \int f(\si) \exp \Ll( \sqrt{2t} H_N(\sigma) - N t\xi \Ll( N^{-1} |\si|^2 \Rr)   + \sum_{\ell = 0}^k \Ll(2q_{\ell} - 2q_{\ell-1}\Rr)^\frac 1 2 y_\ell \cdot \sigma - q_k |\sigma|^2  \Rr) \, \d P_N(\sigma).
\end{multline*}
Defining $F_N(t,\zeta,q)$ as in \eqref{e.barFN.zq} (substituting $\bar F_N$ by $F_N$ there), we will show that for every $\ell \in \{0,\ldots,k\}$,
\begin{multline}
\label{e.pos.dr_q.precise}
\dr_{q_\ell} \E \Ll[ F_N(t,\zeta,q)  \vb   (H_N(\sigma))_{\sigma \in \{\pm 1\}^N} \Rr] 
\\
= \frac {\zeta_{m+1} - \zeta_m} N \sum_{i = 1}^N \E_{y} \Ll[ \Ll(\E_{y_{\ge m+1}} \Ll[\la \si_i \ra D_{m+1,k} \Rr]\Rr)^2 D_{1m} \Rr].
\end{multline}
This clearly implies the lemma. We decompose the proof of \eqref{e.pos.dr_q.precise} into two steps.

\smallskip

\emph{Step 1.} We show that, for every $\ell,m \in \{0,\ldots,k\}$,
\begin{equation}  
\label{e.calc.dr_q}
\dr_{q_{m}} X_{\ell-1} = \E_{y_{\ge \ell}} \Ll[ \Ll(\dr_{q_m} X_k\Rr) D_{\ell k} \Rr].
\end{equation}
We prove the result by decreasing induction on $\ell$. By \eqref{e.def.Xell}, the result is clear for $\ell = k$. Let $\ell \in \{1,\ldots,k\}$, and assume that the statement \eqref{e.calc.dr_q} holds with $\ell$ replaced by $\ell + 1$. Using \eqref{e.def.Xell} again, we obtain \eqref{e.calc.dr_q} itself. This proves \eqref{e.def.Xell} for every $\ell \in \{1,\ldots, k\}$. The statement for $\ell = 0$ is then immediate (recall that $\zeta_0 = 0$). Similarly, for every $\ell ,m \in \{0,\ldots,k\}$ with $m > \ell$ and $i \in \{1,\ldots, N\}$, we have
\begin{equation}  
\label{e.calc.dr_y}
\dr_{y_{mi}} X_{\ell-1} = \E_{y_{\ge \ell}} \Ll[ \Ll(\dr_{y_{mi}} X_k \Rr) D_{\ell k} \Rr],
\end{equation}
where we write $y_m = (y_{mi})_{1 \le i \le N} \in \R^N$. For $m \le \ell$, we clearly have $\dr_{y_{mi}} X_{\ell -1} = 0$.

\smallskip

\emph{Step 2.}
Notice that, for every $m \in \{0,\ldots,k-1\}$,
\begin{equation}  
\label{e.dr_qX}
\dr_{q_m} X_k = \la (2q_m - 2q_{m-1})^{-\frac 1 2} y_m \cdot \sigma - (2q_{m+1}-2q_m)^{-\frac 1 2} y_{m+1} \cdot \si \ra.
\end{equation}
We are ultimately interested in understanding $\dr_{q_m} X_{-1}$, which, in view of \eqref{e.calc.dr_q}, prompts us to study, for every $i \in \{1,\ldots,N\}$,
\begin{equation}  
\label{e.drq.ibp}
\E_y \Ll[ y_{mi}  \la \si_i \ra D_{1k} \Rr] = \E_y \Ll[ \dr_{y_{mi}} \Ll(\la \si_i \ra D_{1k} \Rr)\Rr],
\end{equation}
where we performed a Gaussian integration by parts to get the equality. We have
\begin{equation}  
\label{e.dry.easy}
\dr_{y_{mi}} \la \si_i \ra = (2q_{m} - 2q_{m-1})^\frac 12 \Ll(\la \si_i^2 \ra - \la \si_i \ra^2 \Rr),
\end{equation}
and 
\begin{equation*}  %\label{e.}
\dr_{y_{mi}} D_{1k} 
 = \Ll(\sum_{\ell = m}^k \zeta_\ell \dr_{y_{mi}} X_{\ell} 
     - \sum_{\ell = m+1}^k \zeta_{\ell} \frac{\E_{y_{\ell}} \Ll[\dr_{y_{mi}} X_{\ell} \exp \Ll( \zeta_\ell X_\ell \Rr)   \Rr]}{\E_{y_\ell} \Ll[ \exp \Ll( \zeta_\ell X_\ell \Rr)  \Rr] }\Rr) D_{1k}.
\end{equation*}
We next derive from \eqref{e.calc.dr_y} that, for every $\ell,m \in \{0,\ldots,k\}$ with $m > \ell$ and $i \in \{1,\ldots, N\}$,
\begin{equation*}
%\label{e.dry-expl}
\dr_{y_{mi}} X_{\ell-1} = (2q_m - 2q_{m-1})^\frac 12 \, \E_{y_{\ge \ell}} \Ll[ \la \si_i \ra D_{\ell k} \Rr] .
\end{equation*}
It thus follows that
\begin{equation*}  %\label{e.}
\frac{\E_{y_{\ell}} \Ll[\dr_{y_{mi}} X_{\ell} \exp \Ll( \zeta_\ell X_\ell \Rr)   \Rr]}{\E_{y_\ell} \Ll[ \exp \Ll( \zeta_\ell X_\ell \Rr)  \Rr] } 
= (2q_m - 2q_{m-1})^\frac 1 2 \, \E_{y_{\ge \ell}} \Ll[ \la \si_i \ra D_{\ell k} \Rr] ,
\end{equation*}
and 
\begin{align*}  %\label{e.}
\dr_{y_{mi}} D_{1k} 
& =  (2q_m - 2q_{m-1})^\frac 1 2 \Ll(\sum_{\ell = m}^k \zeta_\ell \E_{y \ge \ell+1} \Ll[ \la \si_i \ra D_{\ell+1,k} \Rr]  - \sum_{\ell = m+1}^k \zeta_{\ell} \E_{y_{\ge \ell}} \Ll[ \la \si_i \ra D_{\ell k} \Rr] \Rr) D_{1k}
\\
& = (2q_m - 2q_{m-1})^\frac 1 2 \Ll(\la \si_i \ra - \sum_{\ell = m}^k  (\zeta_{\ell+1} - \zeta_\ell) \E_{y_{\ge \ell+1}} \Ll[ \la \si_i \ra D_{\ell+1,k} \Rr] \Rr) D_{1k},
\end{align*}
with the understanding that $\E_{y_{\ge k+1}}$ is the identity map, $D_{k+1,k} = 1$, and recalling that $\zeta_{k+1} = 1$. Combining this with \eqref{e.drq.ibp} and \eqref{e.dry.easy}, we thus get that
\begin{multline*}  %\label{e.}
(2q_m - 2q_{m-1})^{-\frac 1 2} \E_y \Ll[ y_{mi}  \la \si_i \ra D_{1k} \Rr] 
\\
=  \E_y \Ll[ \Ll(\la \si_i^2 \ra - \la \si_i \ra \sum_{\ell = m}^k  (\zeta_{\ell+1} - \zeta_\ell) \E_{y_{\ge \ell+1}} \Ll[ \la \si_i \ra D_{\ell+1,k} \Rr] \Rr) D_{1k}\Rr] .
\end{multline*}
Using this identity in conjunction with \eqref{e.calc.dr_q} and \eqref{e.dr_qX}, we arrive at
\begin{equation*}  %\label{e.}
-\dr_{q_{m}} X_{-1} = (\zeta_{m+1} - \zeta_m) \sum_{i = 1}^N\E_y \Ll[  \E_{y_{\ge m+1}} \Ll[ \la \si_i \ra D_{m+1,k} \Rr] \la \si_i \ra D_{1k} \Rr] .
\end{equation*}
This identity is also valid when $m = k$, as can be checked by following the same argument.
We can then write $D_{1k} = D_{1m} D_{m+1,k}$, and use that $D_{1m}$ does not depend on $y_{m+1},\ldots,y_k$, to conclude that
\begin{equation*}
-\dr_{q_{m}} X_{-1} = (\zeta_{m+1} - \zeta_m) \sum_{i = 1}^N \E_{y} \Ll[ \Ll(\E_{y_{\ge m+1}} \Ll[\la \si_i \ra D_{m+1,k} \Rr]\Rr)^2 D_{1m} \Rr].
\end{equation*}
By Lemma~\ref{l.integr.PD}, this is \eqref{e.pos.dr_q.precise}.
\end{proof}

We can now state the conjecture generalizing Theorem~\ref{t.main}.

\begin{conjecture} 
\label{c.main}
Assume \eqref{e.ass.support} and that there exists a function $\psi : \mcl P_*(\R_+) \to \R$ such that for every $\mu \in \mcl P_*(\R_+)$, $\bar F_N(0,\mu)$ converges to $\psi(\mu)$ as $N$ tends to infinity. For every $t \ge 0$ and $\mu \in \mcl P_*(\R_+)$, we have
\begin{equation*}  %\label{e.}
\lim_{N\to +\infty} \bar F_N(t,\mu) = f(t,\mu),
\end{equation*}
where $f : \R_+ \times \mcl P(\R_+) \to \R$ solves the Hamilton-Jacobi equation in \eqref{e.hj}.
\end{conjecture}
Recall that for the purposes of the present paper, we take the Hopf-Lax formula~\eqref{e.hopflax} as the definition of the solution to \eqref{e.hj}. 

%Notice that, by Gaussian integration by parts, see \cite[Lemma~1.1]{pan},
%\begin{equation*}  %\label{e.}
%\dr_t \bar F_N(t,\mu) 
%= - \frac 1 N \E \la \frac 1 {\sqrt{2t}} H_N(\sigma) - Nt \xi \Ll( N^{-1} |\si|^2 \Rr) \ra =  \E \la \xi \Ll( \frac{\si \cdot \si'}{N} \Rr) \ra.
%\end{equation*}
%Under the assumption that the support of the measure $P_N$ is contained in the ball $\{\sigma \in \R^N \ : \ |\sigma|^2 \le N\}$, we can thus combine this identity with Proposition~\ref{p.continuity} and see that the function $f$ in Conjecture~\ref{c.main} must be Lipschitz in both variables, or more precisely, for every $s,t \ge 0$ and $\mu,\nu \in \mcl P_*(\R_+)$,
%\begin{equation}
%\label{e.lipschitz}
%\Ll| f(t,\nu) - f(s,\mu) \Rr| \le \xi(1) |t-s| + \E \Ll[ \Ll| X_{\nu} - X_{\mu} \Rr|  \Rr] .
%\end{equation}

%
%
%
%%%%%%%%%%%%%%%%%%%%%%%%%%%%
%%%%%%%%%%%%%%%%%%%%%%%%%%%%
%
%
%

\section{Convergence of initial condition}
\label{s.conv.init}

We now give two typical situations in which the convergence of $\bar F_N(0,\cdot)$ to some limit is valid. Whenever the limit exists, we write, for every $\mu \in \mcl P_*(\R_+)$,
\begin{equation}
\label{e.def.psi}
\psi(\mu) := \lim_{N\to \infty} \bar F_N(0,\mu).
\end{equation}
In agreement with Conjecture~\ref{c.main}, the function $\psi$ is the initial condition we need to use for the Hamilton-Jacobi equation~\eqref{e.hj}.
\begin{proposition}[Convergence of initial condition]
\label{p.converge.psi}
(1) If the measure $P_N$ is of the product form $P_N = P_1^{\otimes N}$, with, say, $P_1$ of bounded support, then $\bar F_N(0,\cdot) = \bar F_1(0,\cdot)$.

\smallskip

(2) For every $\mu \in \mcl P(\R_+)$ of compact support and $q \ge 0$ such that $\mu([0,q]) = 1$, let
\begin{multline}  
\label{e.def.psi.circ}
\psi^\circ (\mu) := \\
 \inf  \Ll\{\int_0^q \frac 1 {b-2\int_s^{q} \mu([0,r]) \, \d r} \, \d s +  \frac 1 2 \Ll(b - 1 - \log b\Rr) - q \ : \ b > 2\int_0^q \mu([0,r]) \, \d r \Rr\}.
\end{multline}
The right side of \eqref{e.def.psi.circ} does not depend on the choice of $q$ satisfying $\mu([0,q]) = 1$, and the mapping $\mu \mapsto \psi^\circ(\mu)$ can be extended by continuity to $\mcl P_1(\R_+)$. Moreover, if the measure $P_N$ is the uniform measure on the sphere $\{\si \in \R^N \ : \ |\si^2| = N\}$, then for every $\mu \in \mcl P_1(\R_+)$, we have
\begin{equation}  
\label{e.conv.psi.circ}
\lim_{N \to \infty} \bar F_N(0,\mu) = \psi^\circ(\mu).
\end{equation}
\end{proposition}
\begin{proof}
For part (1), we appeal to Lemma~\ref{l.integr.PD} and observe that, when $t = 0$, the definition of $X_k$ given there becomes
\begin{equation*}  %\label{e.}
X_k(y_0,\ldots, y_k) = \sum_{i = 1}^N \log \int_{\R} \exp \Ll( \sum_{\ell = 0}^k \Ll( 2 q_{\ell} - 2 q_{\ell-1} \Rr)^\frac 1 2 y_{\ell,i} \si_i - q_k \si_i^2\Rr)  \, \d P_1(\si_i) .
\end{equation*}
Notice that the summands indexed by $i$ are independent random variables under~$\E_{y_k}$, and this structure is preserved as we go down the levels, up to the definition of $X_0$, where we end up with a sum of $N$ terms that are deterministic and all equal to a constant which does not depend on $N$. This proves the claim (see also \cite[(2.60)]{pan}).

\smallskip

For part (2), we first verify that the right side of \eqref{e.def.psi.circ} does not depend upon the choice of $q \ge 0$ satisfying $\mu([0,q]) = 1$. Indeed, for every $q$ satisfying $\mu([0,q]) = 1$, $q' \ge q$ and $b > 2 \int_0^{q'} \mu([0,r]) \, \d r$, we have
\begin{multline*}  %\label{e.}
\int_0^{q'} \frac 1 {b-2\int_s^{q'} \mu([0,r]) \, \d r} \, \d s 
\\ 
= \int_0^q \frac 1 {b-2(q'-q)-2\int_s^{q} \mu([0,r]) \, \d r} \, \d s  
+ 
\frac 1 2 \Ll(\log b - \log \Ll[ b-2(q'-q) \Rr]\Rr).
\end{multline*}
We thus obtain that
\begin{multline*}  %\label{e.}
\int_0^{q'} \frac 1 {b-2\int_s^{q'} \mu([0,r]) \, \d r} \, \d s + \frac 1 2 \Ll( b-1-\log b \Rr) - q' \\
= \int_0^{q} \frac 1 {b-2(q'-q) - 2\int_s^{q} \mu([0,r]) \, \d r} \, \d s + \frac 1 2 \Ll( b-2(q'-q) - 1-\log \Ll[ b-2(q'-q) \Rr]  \Rr) - q.
\end{multline*}
Taking the infimum over $b > 2 \int_0^{q'} \mu([0,r]) \, \d r = 2(q'-q) + 2 \int_0^q \mu([0,r]) \, \d r$ concludes the verification of the fact that the right side of \eqref{e.def.psi.circ} does not depend on the choice of $q$ satisfying $\mu([0,q]) = 1$. 

\smallskip

In order to verify the convergence in \eqref{e.conv.psi.circ}, we start by considering the case of a measure of finite support. In this case, we can follow the arguments leading to \cite[Proposition~3.1]{tal.sph} and obtain \eqref{e.conv.psi.circ}. The full result then follows by the continuity property of $\bar F_N$, see Proposition~\ref{p.continuity}.
\end{proof}

It so happens that in the case when $P_N$ is the uniform measure on $\{-1,1\}^N$, the initial condition $\psi = \lim_{N \to \infty} \bar F_N(0,\cdot) = \bar F_1(0,\cdot)$ can itself be described in terms of a Hamilton-Jacobi equation of second order \cite{parisi80}. We recall this fact in the proposition below for completeness, and so as to clarify the small modifications necessary to match the different presentation explored in the present paper. As far as I understand, the fact that the initial condition admits such a representation seems to be unrelated to the (first-order) Hamilton-Jacobi structure explored in the rest of the paper. Notice that the possibility to find an equation posed in $\R_+\times \R$ to describe $\psi$, as opposed to a higher-dimensional space, relies crucially on the fact that the norm of $\sigma$ is deterministic. This representation also makes it less transparent that $\psi \ge 0$. 
%
%However, the fact that $\psi$ admits such a representation seems rather accidental to me, and I do not expect it to generalize well to other settings. Indeed, this representation relies crucially on the fact that the norm of $\sigma$ is deterministic (which in classical models only occurs if $\sigma$ is a vector of independent $\pm 1$ random variables, or for the spherical model). (Of course, when for instance we write $-Nt$ in \eqref{e.main}, we also use that $\sigma$ is a vector of independent $\pm 1$ random variables. The generalization here consists in replacing this term by $-Nt\xi(N^{-1} |\si|^2)$.)

\begin{proposition}[Initial condition for Ising spins]
\label{p.init}
Assume that $P_N$ is the uniform measure on $\{-1,1\}^N$, and denote $\psi = \lim_{N \to \infty} \bar F_N(0,\cdot) = \bar F_1(0,\cdot)$. For every $\mu \in \mcl P(\R_+)$ with compact support and $q \ge 0$ such that $\mu([0,q]) = 1$, letting $u_\mu : [0,q]\times \R \to \R$ be the solution of 
\begin{equation}
\label{e.pde.init}
\Ll\{
\begin{aligned}
& \dr_s u_\mu  + \dr_x^2 u_\mu - \mu([0,s]) \, (\dr_x u_\mu)^2 + 1 = 0 & \quad \text{on } [0,q] \times \R,\\
& u_\mu(q,\cdot) = -\log \cosh & \quad \text{on } \R,
\end{aligned}
\Rr.
\end{equation}
we have $\psi(\mu) = u_\mu(0,0)$.
\end{proposition}
\begin{proof}
We first verify that $u_\mu$ does not depend on the choice of $q \ge 0$ satisfying $\mu([0,q]) = 1$. More precisely, denoting by $u_{\mu,q}$ the solution obtained for a given choice of such $q$, and letting $q' \ge q$, we have that the solutions $u_{\mu,q}$ and $u_{\mu,q'}$ coincide on $[0,q] \times \R$. Indeed, this is a consequence of the fact that, with $\phi := - \log \cosh$, we have
\begin{equation*}  %\label{e.}
\dr_x^2 \phi - (\dr_x \phi)^2 + 1 = 0.
\end{equation*}
%and we can thus replace the terminal time condition $u_\mu(+\infty,\cdot) = \log \cosh$ by the condition $u_\mu(S,\cdot) = \log \cosh$, for any given $S$ such that $\mu([0,S]) = 1$. (And this gives the terminal time condition at infinity in \eqref{e.pde.init} a precise meaning.)
%
Let $\mu$ be a measure of the form \eqref{e.def.zeta.q}-\eqref{e.mu.diracs}, and let $(B_t)$ be a standard Brownian motion. We define, for every $x \in \R$,
\begin{equation*}  %\label{e.}
v_\mu(q_k,x) := \log \cosh (x),
\end{equation*}
and then recursively, for every $\ell \in \{0,\ldots,k\}$ and $s \in [q_{\ell-1},q_{\ell})$,
\begin{equation*}  %\label{e.}
v_\mu(s,x) := \zeta_{\ell}^{-1} \log \E \exp \Ll[ \zeta_{\ell} v_\mu\Ll(q_{\ell}, B_{2q_{\ell}} - B_{2s} + x \Rr) \Rr]  -(q_{\ell} - s).
\end{equation*}
Recall that when $\ell = 0$, we have $\zeta_0 = 0$ and we interpret the right side above as
\begin{equation*}  %\label{e.}
\E \Ll[ v_\mu(q_0,B_{2q_1} - B_{2s} + x) \Rr] - (q_0-s).
\end{equation*}
By induction, we have that for every $\ell \in \{0,\ldots, k\}$, 
\begin{align*}  %\label{e.}
& v_\mu(q_{\ell-1},x) \\
& = \zeta_\ell^{-1} \log \E_{y_{\ell}} \Ll[ \E_{y_{\ell+1}}^{\frac{\zeta_{\ell}}{\zeta_{\ell+1}}} \Ll[ \E_{y_{\ell+2}}^{\frac{\zeta_{\ell+1}}{\zeta_{\ell+2}}} \Ll[ \cdots \E_{y_{k}}^{\frac{\zeta_{k-1}}{\zeta_k}} \Ll[ \cosh^{\zeta_k} \Ll( x + \sum_{\ell' = \ell}^k \Ll( 2q_{\ell} - 2q_{\ell-1} \Rr) ^\frac 1 2 y_\ell \Rr)  \Rr] \cdots  \Rr]  \Rr]  \Rr] 
\\
& \qquad - (q_{k}-q_{\ell-1}),
\end{align*}
where here $\E_{y_\ell}$ denotes the integration of the variable $y_\ell$ according to the standard scalar Gaussian measure, and for $\ell = 0$, the right side above is interpreted as
\begin{equation*}  %\label{e.}
\E_{y_0} \Ll[ \zeta_{1}^{-1}\log \Ll(  \E_{y_1}\Ll[\E_{y_2}^{\frac{\zeta_1}{\zeta_2}} \Ll[\cdots\Rr]\Rr] \Rr)  \Rr] - q_k. 
\end{equation*}
By Lemma~\ref{l.integr.PD} with $t = 0$ and $N = 1$, we deduce that $\bar F_1(0,\mu) = - v_\mu(0,0)$, and we have already seen in part (1) of Proposition~\ref{p.converge.psi} that $\bar F_N(0,\cdot) = \bar F_1(0,\cdot)$. Moreover, denoting, for every $\ell \in \{0,\ldots,k\}$ and $s \in [q_{\ell-1},q_\ell)$,
\begin{equation*}  %\label{e.}
w_\mu(s,x) := \exp \Ll( \zeta_\ell (s-q_\ell) \Rr) \E \exp \Ll[ \zeta_\ell v_\mu\Ll(q_{\ell}, B_{2q_{\ell}} - B_{2s} + x \Rr) \Rr],
\end{equation*}
we have $\dr_s w_\mu +\dr_x^2w_\mu - \zeta_\ell w_\mu = 0$ on $[q_{\ell -1}, q_\ell) \times \R$, with continuity at the junction times $s \in \{q_0,\ldots,q_{k}\}$, and a change of variables then gives that $-v_\mu$ solves~\eqref{e.pde.init}. This shows that Proposition~\ref{p.init} holds whenever $\mu$ is a measure of finite support. The general case can then be obtained by continuity (the continuity of $\bar F_1(0,\cdot)$ is a consequence of Proposition~\ref{p.continuity}; for the continuity of $\mu \mapsto u_\mu(0,0)$, one can start by verifying that $\|\dr_x u_\mu\|_{L^\infty} \le 1$ using the maximum principle).
\end{proof}

%
%
%
%%%%%%%%%%%%%%%%%%%%%%%%%%%%
%%%%%%%%%%%%%%%%%%%%%%%%%%%%
%
%
%

\section{Proof of Theorem~\ref{t.main}}
\label{s.proof}

In this section, we give the proof of Theorem~\ref{t.main}. Recall that we interpret the solution of \eqref{e.hj} as being given by the Hopf-Lax formula in \eqref{e.hopflax}. The formula \eqref{e.hopflax} simplifies slightly in the case when $\mu = \de_0$, and thus the statement of Theorem~\ref{t.main} can be reformulated as follows.
\begin{proposition}[Hopf-Lax representation of Parisi formula]
\label{p.hopf-lax}
Assume \eqref{e.ass.sphere}, and fix the normalization $\xi(1) = 1$. For every $t > 0$, we have
\begin{multline*}  
%\label{e.hopf-lax}
\lim_{N \to \infty} -\frac 1 N\E \log\int \exp \Ll( \sqrt{2t} H_N(\sigma) - N t\Rr)   \, \d P_N(\sigma) \\
= \sup_{\mu \in \mcl P_{*}(\R_+)} \Ll( \psi(\mu) - t \int_{\R_+} \xi^*(t^{-1} s) \, \d \mu(s) \Rr) .
\end{multline*}

\end{proposition}

\begin{proof}
We first focus on the case when $P_N$ is the uniform probability measure on~$\{-1,1\}^N$. We decompose the argument for this case into four steps.

\smallskip

\emph{Step 1.} In this step, we recast the standard expression for Parisi's formula, borrowed from~\cite{pan}, in the following form: 
\begin{multline}  
\label{e.borrow}
\lim_{N \to \infty} -\frac 1 N \E \log \int \exp \Ll( \sqrt{2t} H_N(\sigma) -Nt\Rr)    \, \d P_N(\sigma) 
\\
=t + \sup_{\nu \in \mcl P(\Ll[0,1\Rr])}\Ll( \psi \Ll( (t\xi')(\nu) \Rr) -t\xi'(1) + t \int_0^1 s \xi''(s) \nu([0,s]) \, \d s \Rr) .
\end{multline}
On the right side, the notation $(t\xi')(\nu)$ denotes the image of the measure $\nu$ under the mapping $r \mapsto t\xi'(r)$. Let $\nu \in \mcl P([0,1])$ be a measure with finite support containing the extremal points $0$ and $1$. For some $k \in \N$ and parameters
\begin{equation*}  
0 = \zeta_{0} < \zeta_1 < \cdots < \zeta_{k} < \zeta_{k+1} = 1, \qquad 0 = q_{-1} = q_0 < q_1 < \cdots < q_k = 1,
\end{equation*}
we can represent this measure as
\begin{equation*}  
%\label{e.mu.diracs}
\nu = \sum_{\ell = 0}^{k} (\zeta_{\ell+1} - \zeta_{\ell}) \de_{q_\ell}.
\end{equation*}
The reason for the perhaps slightly surprising choice of setting $q_{-1} = q_0 = 0$ is that we have chosen here to include a term associated with the root of $\mcl A$, at level $\ell = 0$, in the definition \eqref{e.def.HNalpha}, while a different choice was taken in \cite{pan}. (The motivation for this inconsequential difference is that it then covers more naturally the situation in~\eqref{e.def.FNcirc} as a particular case. Relatedly, by default, the measures of finite support considered in \cite{pan} have an atom at zero.) 
In order to extract the free energy associated with the Hamiltonian $\sigma \mapsto \sqrt{2t} H_N(\sigma)$ from \cite[Theorem~3.1]{pan}, we need to replace $\xi$ by $2t\xi$ in \cite[(3.3)]{pan}. With this modification in place, and recalling Lemma~\ref{l.integr.PD}, we see that the quantity denoted $\E X_0$ in \cite[(3.11)]{pan} can be rewritten as
\begin{equation*}  %\label{e.}
\E \log \Ll[  \sum_{\sigma \in \{\pm 1\}} \sum_{\al \in \mcl A} v_\al \exp \Ll( \sum_{\ell = 1}^k \Ll(2t\xi'(q_{\ell}) - 2t\xi'(q_{\ell-1})\Rr)^\frac 1 2 z_{\al_{|\ell}} \cdot \sigma \Rr)\Rr].
\end{equation*}
On the other hand, by \eqref{e.def.psi} and Proposition~\ref{p.converge.psi}, we have 
\begin{multline*}  %\label{e.}
\psi\Ll((t\xi')(\nu)\Rr) \\
= -\E \log \Ll[ \frac 1 2 \sum_{\sigma \in \{\pm 1\}} \sum_{\al \in \mcl A} v_\al \exp \Ll( \sum_{\ell = 1}^k \Ll(2t\xi'(q_\ell) - 2t\xi'(q_{\ell-1})\Rr)^\frac 1 2 z_{\al_{|\ell}} \cdot \sigma -t\xi'(1)\Rr)\Rr].
\end{multline*}
We thus deduce that the quantity denoted $\E X_0$ in \cite[(3.11)]{pan} is
\begin{equation*}  %\label{e.}
\log 2 - \psi\Ll((t\xi')(\nu)\Rr)  + t \xi'(1).
\end{equation*}
The finite-volume free energy is normalized slightly differently here and in \cite{pan}: there is a multiplicative factor of $2^{-N}$ hidden in the fact that $P_N$ is normalized to be a probability measure, and an additional minus sign, on the left side of~\eqref{e.borrow}. Combining these observations and appealing to \cite[Theorem~3.1]{pan} and to Proposition~\ref{p.continuity} yields~\eqref{e.borrow}.

\smallskip

%Notice that $\xi'$ is a bijection from $\R_+$ onto itself. We can thus define the probability measure $\nu \in \mcl P(\R_+)$ to be the image of $\mu$ under the mapping $(\xi')^{-1}$. For every $f \in C_b(\R_+)$, we have
%\jccomment{this is a bit ridiculous}
%\begin{equation}  
%\label{e.change.var}
%\int_{\R_+} f(\xi'(s)) \,\d \nu(s) = \int_{\R_+} f(s) \, \d \mu(s).
%\end{equation}

%\smallskip

\emph{Step 2}. We fix $\nu \in \mcl P([0,1])$, $t > 0$, and define $\mu := (t\xi')(\nu)$ to be the image of $\nu$ under the mapping~$r \mapsto t\xi'(r)$. 
%Notice that $(t\xi')(\nu) = \mu_t$, where $\mu_t$ is the image of $\mu$ under the mapping $q \mapsto tq$.
In this step, we show that
\begin{equation}
\label{e.explicit.W}
 \int_{ \Ll[0,1\Rr] } \Ll( s\xi'(s) - \xi(s) \Rr)  \, \d \nu(s) = \int_{\R_+} \xi^*(t^{-1} s) \, \d \mu(s).
\end{equation}
%Indeed, by \cite[Theorem~2.18 and Remark~2.19(ii)]{villani} (or \cite[Theorem~6.0.2]{AGS}), we have
%\begin{equation*}  %\label{e.}
%\mcl T_*(\mu,\mu') = \int_0^1 \xi^* \Ll(  F_\mu^{-1}(r)  - F_{\mu'}^{-1}(r)\Rr) \, \d r , 
%\end{equation*}
%where we write, for every $r \in [0,1]$,
%\begin{equation}
%\label{e.def.F-1}
%F^{-1}_\mu(r) := \inf \Ll\{ s \in \R \ : \ \mu \Ll( (-\infty,s] \Rr)  > r \Rr\} .
%\end{equation}
%Using also that the image of the Lebesgue measure on $[0,1]$ under the mapping $F_\mu^{-1}$ is the measure $\mu$, we obtain that
%\begin{align}  
%\label{e.write.W*}
%\mcl T_*(\de_0,\mu)   = \int_0^1 \xi^* \Ll( F_\mu^{-1}(r) \Rr) \, \d r
% = \int_{\R_+} \xi^*(s) \, \d \mu(s).
%\end{align}
By the definition of $\mu$ and a change of variables, we have
\begin{equation*}  %\label{e.}
\int_{\R_+} \xi^*(t^{-1} s) \, \d \mu(s) = \int_{\Ll[0,1\Rr]} \xi^*(\xi'(s)) \, \d \nu(s).
\end{equation*}
Recall that
\begin{equation*}  %\label{e.}
\xi^*(\xi'(s)) = \sup_{r \ge 0} \Ll(  r\xi'(s) - \xi(r)\Rr).
\end{equation*}
Since $\xi'(0) = 0$, for each $s > 0$, the supremum above is achieved at some $r > 0$, and calculating the derivative in $r$ shows that it is in fact achieved at $r = s$, since $\xi'$ is injective. That is, we have $\xi^*(\xi'(s)) = s\xi'(s) - \xi(s)$, and thus \eqref{e.explicit.W} holds.

\smallskip
\emph{Step 3.}
In this step, we show that
\begin{multline}  
\label{e.almost.hopf-lax}
\lim_{N \to \infty} -\frac 1 N\E \log\int \exp \Ll( \sqrt{2t} H_N(\sigma) - N t\Rr)   \, \d P_N(\sigma) \\
= \sup \Ll \{ \psi(\mu) - t \int_{\R_+} \xi^*(t^{-1} s) \, \d \mu(s) \ : \ \mu \in \mcl P(\R_+), \ \supp \mu \subset [0,t\xi'(1)] \Rr\} .
\end{multline}
We start by rewriting the last term in the supremum on the right side of \eqref{e.borrow}, by appealing to the following integration by parts formula: for every $f \in L^1([0,1])$,
\begin{equation}
\label{e.ibp}
\int_0^1 f(r) \, \nu([0,r]) \, \d r = \int_{ \Ll[ 0,1 \Rr] } \int_r^1 f(u) \, \d u \, \d \nu(r).
\end{equation}
This formula itself is a consequence of Fubini's theorem.
We notice that
\begin{equation*}  %\label{e.}
\int_r^1 s \xi''(s) \, \d s = \xi'(1) - \xi(1)- \Ll(r\xi'(r)  - \xi(r)\Rr).
\end{equation*}
Recalling that we have fixed the normalization $\xi(1) = 1$ yields that
\begin{equation*}  %\label{e.}
\int_0^1 s \xi''(s) \nu([0,s]) \, \d s 
= \xi'(1) - 1  - \int_{ \Ll[ 0,1 \Rr] } \Ll( r \xi'(r) - \xi(r) \Rr) \, \d \nu(r).
\end{equation*}
Combining this with \eqref{e.borrow}, \eqref{e.explicit.W}, and the fact that $\xi' : [0,1] \to [0,\xi'(1)]$ is bijective, we obtain \eqref{e.almost.hopf-lax}.

\smallskip

\emph{Step 4.} In order to conclude the proof (in the case of Ising spins), there remains to show that the supremum on the right side of \eqref{e.almost.hopf-lax} does not increase if we remove the restriction that the support of the measure $\mu$ be in $[0,t \xi'(1)]$. Let $\mu \in \mcl P_*(\R_+)$, and let $\td \mu$ denote the image of $\mu$ under the mapping $r \mapsto r \wedge (t\xi'(1))$, where we write $a \wedge b := \min(a,b)$.
We show that
\begin{equation}  
\label{e.localize.mu}
\psi(\mu) - t \int_{\R_+} \xi^*(t^{-1} s) \, \d \mu(s) 
\le 
\psi(\td \mu) - t \int_{\R_+} \xi^*(t^{-1} s) \, \d \td \mu(s) .
\end{equation}
By Proposition~\ref{p.continuity} and Fubini's theorem, we have
\begin{align}  %\label{e.}
\notag
\psi(\mu) - \psi(\td \mu) 
 \le \E \Ll[ \Ll| X_\mu - X_{\td \mu} \Rr|  \Rr]  
 & = \int_0^{+\infty} \Ll|\mu([0,r]) - \td \mu([0,r]) \Rr| \, \d r
\\
\notag
&  = \int_{t\xi'(1)}^{+\infty} \mu((r,+\infty)) \, \d r 
\\
\notag
& =  \int_{t\xi'(1)}^{+\infty} \int_{\R_+}\1_{s \ge r} \, \d \mu(s) \, \d r 
\\
\label{e.first.bit}
& = \int_{\R_+} \Ll(s - t\xi'(1) \Rr)\1_{s \ge t\xi'(1)}\, \d \mu(s).
\end{align}
On the other hand, by the definition of $\td \mu$, we have
\begin{equation*}  %\label{e.}
\int_{\R_+} \xi^*(t^{-1} s) \, \d \td \mu(s)  = \int_{\R_+} \xi^*\Ll((t^{-1} s) \wedge \xi'(1) \Rr) \, \d \mu(s),
\end{equation*}
and thus
\begin{multline}  
\label{e.second.bit}
\int_{\R_+} \xi^*(t^{-1} s) \, \d \mu(s) - \int_{\R_+} \xi^*(t^{-1} s) \, \d \td \mu(s) 
\\
= \int_{\R_+} \Ll(\xi^*(t^{-1}s) - \xi^*\Ll( \xi'(1)\Rr)\Rr)\1_{s \ge t\xi'(1)} \, \d \mu(s).
\end{multline}
Recall that $\xi^*(\xi'(1)) = \xi'(1) - \xi(1)$. By the definition of the convex dual, we also have that $\xi^*(s) \ge s - \xi(1)$. Hence, the integral on the right side of \eqref{e.second.bit} is bounded from below by
\begin{equation*}  %\label{e.}
\int_{̣\R_+} \Ll( t^{-1} s -\xi'(1) \Rr) \1_{s \ge t \xi'(1)} \, \d \mu(s).
\end{equation*}
Combining this with \eqref{e.first.bit} yields \eqref{e.localize.mu} and thus completes the proof in the case of Ising spins.

\smallskip

\emph{Step 5.} We show Proposition~\ref{p.hopf-lax} in the case when $P_N$ is the uniform probability measure on the sphere $\{\si \in \R^N \ : \ |\si|^2 = N\}$. Using \cite[Corollary~4.1]{tal.sph} and arguing as in Step 1, one can check that the formula \eqref{e.borrow} is also valid in this case. The rest of the argument carries over without modification.
\end{proof}

%For each $\mu, \mu' \in \mcl P_{*}(\R_+)$ and $t > 0$, there exists a unique geodesic connecting $\mu$ to $\mu'$ over the time interval $[0,t]$: we denote it by 
%\begin{equation*}  %\label{e.}
%\Ll\{
%\begin{array}{rcl}  %\label{}
%[0,t] & \to & \mcl P_*(\R_+) \\
%s & \mapsto & \gamma_{\mu,\mu',t}^s.
%\end{array}
%\end{equation*}
%Let $\rho$ be an optimal transport plan between $\mu$ and $\mu'$. We have 
%\begin{equation*}  %\label{e.}
%F^{-1}_{\gamma_{\mu,\mu',t}^s} = \Ll( 1 - \frac s t \Rr) F^{-1}_\mu + \frac s t F^{-1}_{\mu'}.
%\end{equation*}
%Recall that 
%\begin{equation*}  %\label{e.}
%\mcl T_*(\mu,\mu') = \int_0^1 \xi^* \Ll( F_\mu^{-1}(s) - F_\nu^{-1}(s) \Rr) \, \d s,
%\end{equation*}
%so the line above is kind of obvious, by convexity of $\xi^*$. We define 
%\begin{equation*}  %\label{e.}
%\xi^*(\dot \gamma_{\mu,\mu',t}) := \int_0^1 \xi^* \Ll( \frac{F_{\mu'}^{-1}(r) - F_{\mu}^{-1}(r)}{t} \Rr) \, \d r.
%\end{equation*}
%Notice that, for every $s \in [0,t]$,
%\begin{equation*}  %\label{e.}
%\xi^*(\dot \gamma_{\mu,\mu',t}) = \lim_{h \to 0} \frac{\mcl T_*\Ll(\gamma^{s+h}_{\mu,\mu',t},\gamma^{s}_{\mu,\mu',t}\Rr)}{h}.
%\end{equation*}

%
%
%%%%%%%%%%%%%%%%%%%%%%%%%%%%
%%%%%%%%%%%%%%%%%%%%%%%%%%%%
%
%
%

\section{Finite-dimensional approximations}
\label{s.finite.dim}

In this last section, we lightly touch upon the question of giving an intrinsic meaning to the Hamilton-Jacobi equation \eqref{e.hj}. This allows to give some substance to the connection between this equation and the Hopf-Lax formula in \eqref{e.hopflax}. 

\smallskip

There already exists a rich literature on Hamilton-Jacobi equations in infinite-dimensional Banach spaces, as well as on the Wasserstein space of probability measures or more general metric spaces; see in particular \cite{cl1} for the former and \cite{gan08, amb14} for the latter. I will refrain from engaging with these works here, and only discuss finite-dimensional approximations of the solution to \eqref{e.hj}.

\smallskip

A simple way to obtain a finite-dimensional approximation of \eqref{e.hj} is to fix an integer $k \ge 1$ and restrict the space of allowed probability measures to those belonging to
\begin{equation*}  %\label{e.}
\mcl P^{(k)}(\R_+) := \Ll\{ \frac 1 {k} \sum_{\ell = 1}^k \de_{x_\ell} \ : \ x_1, \ldots, x_k \ge 0 \Rr\} .
\end{equation*}
%(The reason for the perhaps surprising choice of indexing this set by $k$ instead of $k+1$ is to preserve consistency with the terminology of $k$-step replica symmetry breaking.) 
A natural discretization of the formula \eqref{e.hopflax} is then obtained by setting, for every $t \ge 0$ and $\mu \in \mcl P^{(k)}(\R_+)$,
\begin{equation*}  %\label{e.}
f^{(k)}(t,\mu) := \sup_{\nu \in \mcl P^{(k)} (\R_+)} \Ll( \psi(\nu) - t \E \Ll[ \xi^* \Ll( \frac{X_\nu - X_\mu}{t} \Rr)  \Rr]  \Rr) .
\end{equation*}
Abusing notation slightly, we also write, for every $x \in \R_+^{k}$,
\begin{equation*}  %\label{e.}
f^{(k)}(t,x) := f \Ll( t, \frac 1 {k} \sum_{\ell = 1}^k \de_{x_\ell} \Rr),  \quad \text{ and } \quad \psi^{(k)}(x) := \psi \Ll( \frac 1 {k} \sum_{\ell = 1}^k \de_{x_\ell} \Rr).
\end{equation*}
We note the following elementary observation.
\begin{lemma}
\label{l.rewrite.fk}
For every $t \ge 0$ and $x \in \R_+^{k}$, we have
\begin{equation}  
\label{e.rewrite.fk}
f\k(t,x) = \sup_{y \in \R^{k}_{+}} \Ll( \psi^{(k)}(y) - \frac t {k} \sum_{\ell = 1}^k \xi^* \Ll( \frac{y_\ell - x_\ell}{t} \Rr)  \Rr) .
\end{equation}
\end{lemma}
\begin{proof}
We introduce the notation
\begin{equation*}  %\label{e.}
\R^{k\uparrow}_{+} := \Ll\{ x = (x_1,\ldots,x_k) \in \R_+^k \ : \ x_1 \le \cdots \le x_k \Rr\} .
\end{equation*}
Notice first that the quantities $f\k(t,x)$ and $\psi^{(k)}(x)$ are invariant under permutation of the coordinates of $x$. Hence, it suffices to prove the relation \eqref{e.rewrite.fk} under the additional assumption that $x \in \R^{k\uparrow}_+$. It is clear that equality holds if on the right side, we take the supremum over $y \in \R^{k\uparrow}_+$ only. We now verify that other orderings of a given vector $y$ yield a larger value for the sum on the right side of \eqref{e.rewrite.fk}. Indeed, fix $x \in \R^{k\uparrow}_+$, $y \in \R^k$, and assume that there exist $i < j \in \{1,\ldots,k\}$ such that $y_i \ge y_j$. By the convexity of $\xi^*$ and the fact that $x_i \le x_j$,, the function $u \mapsto  \xi^* \Ll( \frac{u - x_i}{t} \Rr) - \xi^* \Ll( \frac{u - x_j}{t} \Rr)$ is increasing. In particular,
\begin{equation*}  %\label{e.}
 \xi^* \Ll( \frac{y_i - x_i}{t} \Rr) - \xi^* \Ll( \frac{y_i - x_j}{t} \Rr) \ge   \xi^* \Ll( \frac{y_j - x_i}{t} \Rr) - \xi^* \Ll( \frac{y_j - x_j}{t} \Rr),
\end{equation*}
and therefore
\begin{equation*}  %\label{e.}
\xi^* \Ll( \frac{y_j - x_i}{t} \Rr) + \xi^* \Ll( \frac{y_i - x_j}{t} \Rr) \le \xi^* \Ll( \frac{y_i - x_i}{t} \Rr) + \xi^* \Ll( \frac{y_j - x_j}{t} \Rr) .
\end{equation*}
That is, whenever $i < j$ and $y_i \ge y_j$, replacing $y$ by the vector with the coordinates $y_i$ and $y_j$ interchanged can only reduce (or keep constant) the value of the quantity 
\begin{equation*}  %\label{e.}
\sum_{\ell = 1}^k \xi^* \Ll( \frac{y_\ell - x_\ell}{t} \Rr) .
\end{equation*}
By induction, this implies that replacing the vector $y$ by the increasingly ordered sequence of coordinates of $y$ can only reduce (or keep constant) the quantity above. 
\end{proof}
The convex dual of the mapping
\begin{equation*}  %\label{e.}
\Ll\{
\begin{array}{rcl}
\R^k & \to & \R \\
x & \mapsto & \frac 1 k \sum_{\ell = 1}^k \xi^*(x_\ell)
\end{array}
\Rr.
\end{equation*}
is
\begin{equation}  
\label{e.def.dualdual}
\Ll\{
\begin{array}{rcl}
\R^k & \to & \R \\
p & \mapsto & 
\Ll|
\begin{array}{ll}
\frac 1 k \sum_{\ell = 1}^k \xi(k \, p_\ell) & \text{if } p \in \R_+^k \\
+\infty & \text{otherwise}.
\end{array}
\Rr.
\end{array}
\Rr.
\end{equation}
It follows from Proposition~\ref{p.continuity} that $\psi^{(k)}$ is Lipschitz continuous, and from this, one can show that $f^{(k)}$ is Lipschitz continuous in $x$ and in $t$. In particular, the function $f^{(k)}$ is differentiable almost everywhere in $[0,\infty)\times \R_+^k$. Following classical arguments, see e.g.\ \cite{benton} or \cite[Theorem~3.3.5]{evans}, we thus deduce from \eqref{e.rewrite.fk} that at every $(t,x) \in (0,\infty)\times (0,\infty)^k$ at which $f^{(k)}$ is differentiable, we have
\begin{equation}  
\label{e.finite.dim.hj}
\dr_t f^{(k)}(t,x) - \frac  1 k \sum_{\ell = 1}^k \xi \Ll( k \dr_{x_\ell} f^{(k)}(t,x) \Rr) = 0.
\end{equation}
This identification also uses that $\dr_{x_\ell} f^{(k)}(t,x) \ge 0$, see \eqref{e.def.dualdual}. The latter property can be obtained as a consequence of the fact that $\dr_{x_\ell} \psi^{(k)} \ge 0$, which itself follows from~Lemma~\ref{l.pos.dr_q}. 

\smallskip

We can now verify that the equation in \eqref{e.finite.dim.hj} is formally consistent with a finite-dimensional interpretation of the Hamilton-Jacobi equation \eqref{e.hj}. In view of \eqref{e.taylor.measure}, and assuming ``smoothness'' of the function $f$, we must have, for every $\mu = k^{-1} \sum_{\ell = 1}^k \de_{x_\ell} \in \mcl P^{(k)}(\R_+)$ and $\ell \in \{1,\ldots, k\}$ that
\begin{equation*}  %\label{e.}
\dr_{\mu} f(t,\mu,x_\ell) = k\, \dr_{x_\ell} f^{(k)}(t,x),
\end{equation*}
and thus
\begin{equation*}  %\label{e.}
\int \xi(\dr_{\mu} f) \, \d \mu = \frac 1 k \sum_{\ell = 1}^k \xi \Ll(k \,  \dr_{x_\ell} f^{(k)} \Rr) .
\end{equation*}
We have thus obtained a formal relation between the Hamilton-Jacobi equation in~\eqref{e.hj} and that in \eqref{e.finite.dim.hj}, which itself can be rigorously connected with the Hopf-Lax formula \eqref{e.rewrite.fk}---see \cite{HJinfer} and \cite{HJrank} on how to handle the boundary condition on $\dr(\R_+^k)$ in the contexts of viscosity solutions and weak solutions respectively.

%\jc{Do we want to give bounds on $|f-f^{(k)}|$?}

%\jc{How do we verify that RS solutions are selected for small $t$?}

%\jc{Can we understand better what the solution is doing, at least in the spherical case? How are transitions happening? Can we give a rigorous AT-line-style criterion?}

%

\medskip

\noindent \textbf{Acknowledgements.} I would like to thank Dmitry Panchenko for useful comments, in particular for pointing out an error in an earlier argument for the validity of Lemma~\ref{l.pos.dr_q}. I was partially supported by the ANR grants LSD (ANR-15-CE40-0020-03) and Malin (ANR-16-CE93-0003).

\small
\bibliographystyle{abbrv}
\bibliography{parisi}

\end{document}